\documentclass[a4paper,11pt,twoside]{amsart}
\usepackage[english]{babel}
\usepackage[utf8]{inputenc}

\usepackage[a4paper,inner=3cm,outer=3cm,top=4cm,bottom=4cm,pdftex]{geometry}
\usepackage{fancyhdr}
\usepackage{titlesec}
\usepackage{color}
\usepackage{bold-extra}
\titleformat{\section}{\normalfont\scshape\centering}{\thesection}{1em}{}
\titleformat{\subsection}{\bfseries}{\thesubsection}{1em}{}
  
\usepackage{comment}
\usepackage{graphics}
\usepackage{aliascnt}
\usepackage[pdftex,citecolor=green,linkcolor=red]{hyperref}

\usepackage{amsmath}
\usepackage{amsfonts}
\usepackage{amssymb}
\usepackage{amsthm}
\usepackage{comment}
\usepackage{mathtools}

\newtheorem{theorem}{Theorem}[section]
\newtheorem{corollary}[theorem]{Corollary}

\newtheorem{lemma}[theorem]{Lemma}

\newtheorem{proposition}[theorem]{Proposition}
\theoremstyle{definition}
\newtheorem{definition}[theorem]{Definition}
\newtheorem{remark}[theorem]{Remark}

\numberwithin{equation}{section}

\newcommand{\ord}{\textup{ord}}
\newcommand{\Gal}{\textup{Gal}}

\begin{document}

\title{Simultaneous insolvability of exponential congruences}

\date{}
\author{Olli J\"arviniemi}
\address{Department of Mathematics and Statistics, P.O. Box 68, 00014 Helsinki, Finland}
\email{olli.jarviniemi@helsinki.fi}

\begin{abstract}
We determine a necessary and sufficient condition for the infinitude of primes $p$ such that none of the equations $a_i^x \equiv b_i \pmod{p}, 1 \le i \le n,$ are solvable. We control the insolvability of $a^x \equiv b \pmod{p}$ by power residues for multiplicatively independent $a$ and $b$, and by divisibilities and, most importantly, parities of orders in multiplicatively dependent cases. We also consider a more general problem concerning divisibilities of orders. The problems are motivated by Artin's primitive root conjecture and its variants.
\end{abstract}

\maketitle
\section{Introduction}
The famous Artin primitive root conjecture asserts that any integer $a$, not equal to a square or $-1$, is a primitive root modulo $p$ for infinitely many primes $p$. 

It is still an open problem to prove the statement for all $a$. However, considerable progress has been made. Hooley \cite{hooley} famously proved that, under a suitable generalization of the Riemann hypothesis (GRH), the set of primes $p$ for which $a$ is a primitive root modulo $p$ has a density, positive for $a$ not equal to $-1$ or a square. Unconditionally, Heath-Brown \cite{heath-brown} has shown that the statement is true for ``many'' values, for example for all except at most two primes. For a comprehensive survey on Artin's conjecture, see \cite{moree}.

The so-called two-variable Artin conjecture concerns the set of primes $p$ for which the equation $a^x \equiv b \pmod{p}$ is solvable for fixed integers $a$ and $b$. In the multiplicatively dependent case one can show (unconditionally) that the density of such $p$ exists and is a rational number, positive except for cases where this set is trivially finite. The density question has been solved for the multiplicatively independent case assuming GRH. We refer to \cite{ms} for these results.

Let $f(a, b, x)$ denote the number of primes $p \le x$ such that $a^x \equiv b \pmod{p}$ is solvable. While it has been proven that $f(a, b, x)$ is of magnitude $c\pi(x)$ under GRH, where $\pi(x)$ is the number of primes $p \le x$ and $c = c(a, b)$ is a constant, the best unconditional results in this direction are $f(a, b, x) \ge c'\log(x)$ for multiplicatively independent $a, b$, proven recently in \cite{mss}.

One generalization for the Artin conjecture is considering the set of primes $p$ for which all of the integers $a_1, \ldots , a_n$ are primitive roots modulo $p$. This problem has been treated by Matthews in \cite{matthews}, where it is determined when this set is infinite under GRH.

It is thus natural to consider the simultaneous solvability of congruences of the form $a^x \equiv b \pmod{p}$. In \cite{jarviniemi} the following result is proven.
\begin{theorem}
Assume GRH. Let $a_1, \ldots , a_n, b_1, \ldots , b_n$ be integers greater than one. There are infinitely many primes $p$ such that all of the equations $a_i^x \equiv b_i \pmod{p}, 1 \le i \le n$ are solvable. Furthermore, the density of such primes exists and is positive.
\end{theorem}
Some remarks on the case of negative integers are also given in \cite{jarviniemi}.

Here we consider the complement problem: when do there exist infinitely many primes $p$ such that none of the congruences $a_i^x \equiv b_i \pmod{p}$ are solvable?

Schinzel \cite{schinzel} has considered systems of exponential congruences, though he has studied systems of the form $\prod_{1 \le i \le k} a_{j, i}^{x_i} \equiv b_j \pmod{p}, j = 1, 2, \ldots , n$, so the results do not immediatelly apply to our problem. However, his results are enough to solve the case of one equation (which he has also considered in \cite{schinzel2}). Similarly, Somer \cite{somer} has studied the maximal divisors of $k$th order linear recurrences, and while the results solve our problem for one equation, they do not apply to the general case in our problem.

We will see that multiplicatively independent pairs $(a_i, b_i)$ do not affect the simultaneous insolvability of exponential congruences.  In the dependent cases one has to guarantee divisibility of orders by odd primes and parity conditions on orders. The former do not cause any obstructions either, but the parities do. As an easy example, $2$ and $3$ having odd orders modulo a prime implies that $6$ has odd order too.

Motivated by this, we consider the general problem on satisfiability of (in)divisibility of multiplicative orders. We prove that such conditions may be reduced to considering (in)divisibility by prime powers.

This article is structured as follows. We first perform elementary considerations on exponential congruences to turn the problem into a more tractable form. We then state the main results, which we prove after introducing some notation and preliminaries.

The author thanks Joni Teräväinen for helpful discussions and comments on earlier versions of the manuscript, and the referee for a thorough reading of the article, corrections and suggestions.

\section{Solvability of exponential congruences} \label{sec:elementary}

We first characterize the solvability of an equation $a^x \equiv b \pmod{p}$, $a, b \in \mathbb{Z}$. We use the following terminology.

\begin{definition}
We say that a pair of integers $(a, b)$ is...

\begin{itemize}
\item ...\emph{trivial}, if $|a| \le 1$, $b \in \{0, 1\}$, or $b = a^k$ for some $k \ge 1$.
\item ...\emph{irrational}, if $a$ and $b$ are multiplicatively independent over $\mathbb{Q}$.
\item ...\emph{odd}, if $b = -a^k$ for some $k \ge 0$.
\item ...\emph{divisible}, if $b^s = \pm a^r$ for some positive $r, s$ with $\gcd(r, s) = 1$ and $s \ge 2$ not a power of two.
\item ...\emph{even}, if $b^s = a^r$ for some positive $r, s$ with $\gcd(r, s) = 1$ and $s \ge 2$ a power of two.
\item ...\emph{stronly even}, if $b^s = -a^r$ for some positive $r, s$ with $\gcd(r, s) = 1$ and $s \ge 2$ a power of two.  
\end{itemize}
\end{definition}

For a trivial pair $(a, b)$ it is trivial to determine whether the equation $a^x \equiv b \pmod{p}$ is solvable or not: if $b$ is a power of $a$, the equation is always solvable, and otherwise it is insolvable for all except finitely many $p$.

If the pair $(a, b)$ is divisible, even or strongly even, then $b^s = \pm a^r$ for some $\gcd(r, s) = 1$. If $b^s = a^r$, there exists an integer $c$ such that $b = c^r$ and $a = c^s$. If $b^s = -a^r$, there exists an integer $c$ for which $b = c^r$ and $a = -c^s$. The number $c$ is called the \emph{core} of the pair $(a, b)$.

The characterization is done in the following series of lemmas. Here and in what follows $p$ is a prime and $\ord_p(a)$ denotes the order of $a$ modulo $p$. We always assume $p \nmid a$ when using this notation.

\begin{lemma}
\label{ord}
Let $(a, b)$ be a pair of integers, and let $p$ be a prime not dividing $ab$. The equation $a^x \equiv b \pmod{p}$ is solvable if and only if $\ord_p(b) | \ord_p(a)$.
\end{lemma}

\begin{proof}
Let $g$ be a primitive root modulo $p$, and let $g^A \equiv a \pmod{p}$ and $g^B \equiv b \pmod{p}$. The equation $a^x \equiv b \pmod{p}$ is equivalent with
$$Ax - B \equiv 0 \pmod{p-1}.$$
This equation is solvable if and only if $(A, p-1) | B$, which is equivalent to the condition, as $\ord_p(a) = (p-1)/(A, p-1)$ and $\ord_p(b) = (p-1)/(B, p-1)$.
\end{proof}

\begin{lemma}
\label{odd}
Let $(a, b)$ be an odd pair, and let $p \nmid 2ab$ be a prime. The equation $a^x \equiv b \pmod{p}$ has no solution if and only if $\ord_p(a)$ is odd.
\end{lemma}

\begin{proof}
Let $b = -a^k$ with integer $k$. The equation $a^x \equiv b \pmod{p}$ is equivalent to $a^{x - k} \equiv -1 \pmod{p}$. By Lemma \ref{ord} this is insolvable if and only if $2 = \ord_p(-1) \nmid \ord_p(a)$.
\end{proof}

\begin{lemma}
\label{even}
Let $(a, b)$ be an even pair, let $c$ be its core, and let $p \nmid 2ab$ be a prime. The equation $a^x \equiv b \pmod{p}$ is insolvable if and only if $2 \mid \ord_p(c)$.
\end{lemma}

\begin{proof}
Write $a = c^s$ and $b = c^r$ for integers $(r, s) = 1$, and write $a^x \equiv b \pmod{p}$ as $c^{sx - r} \equiv 1 \pmod{p}$. This is insolvable if and only if $sx - r \equiv 0 \pmod{\ord_p(c)}$ is insolvable, which is equivalent to $(s, \ord_p(c)) \nmid r$. By assumption, $s$ is a power of two, so we must have $2 \mid \ord_p(c)$. This is also sufficient for insolvability.
\end{proof}

\begin{lemma}
\label{strong}
Let $(a, b)$ be a strongly even pair, let $c$ be its core, and let $p \nmid 2ab$ be a prime. The equation $a^x \equiv b \pmod{p}$ is insolvable if and only if $2 \mid \ord_p(c^2)$.
\end{lemma}

\begin{proof}
Write $a = -c^s$ and $b = c^r$, so $a^x \equiv b \pmod{p}$ is equivalent with $(-c^s)^x \equiv c^r \pmod{p}$. If $x$ is even, this is equivalent with the insolvability of $c^{sx - r} \equiv 1 \pmod{p}$, which, as in the previous lemma, is equivalent to $2 \mid \ord_p(c)$. If $x$ is odd, this is equivalent with $c^{sx - r} \equiv -1 \pmod{p}$, that is,

$$sx - r \equiv \frac{\ord_p(c)}{2} \pmod{\ord_p(c)}.$$
Using the fact that $s \ge 2$ is a power of two, the insolvability of this is equivalent to $4 \mid \ord_p(c)$, which in turn is equivalent with $2 \mid \ord_p(c^2)$.
\end{proof}

Finally, for divisible pairs we have the following lemma. The proof is so similar to the proofs above that we omit it.

\begin{lemma}
\label{divisible}
Let $(a, b)$ be a divisible pair, and let $c$ be its core. Pick some $p \nmid 2ab$. Write $b = c^r$ and $a = \pm c^s$ with $(r, s) = 1$. The equation $a^x \equiv b \pmod{p}$ is insolvable if (but not only if) $q | \ord_p(c)$ for some odd prime $q | s$.
\end{lemma}

Thus, we have four type of conditions to think about:
\begin{itemize}
\item[(i)] Guarantee the insolvability of $a^x \equiv b \pmod{p}$ for irrational pairs $(a, b)$.
\item[(ii)] For divisible pairs $(a, b)$, guarantee the divisibility of $\ord_p(c)$ by an odd prime $q$, where $c$ is the core of $(a, b)$.
\item[(iii)] For odd pairs $(a, b)$, guarantee the oddness of $\ord_p(c)$, where $c$ is the core of $(a, b)$.
\item[(iv)] For even and strongly even pairs $(a, b)$, guarantee the evenness of $\ord_p(c)$, where $c$ is the core or the square of the core of $(a, b)$.
\end{itemize}

We will see that (i) and (ii) cause no obstructions at all, so we are mainly concerned with (iii) and (iv). We now let $o_1, \ldots , o_O$ denote the integers whose order are required to be odd in (iii) and $e_1, \ldots , e_E$ denote the integers whose orders are required to be even in (iv).

We note right away that if the product $o_1^{x_1} \cdots o_O^{x_O}$ equals $-1$ for some $x_1, \ldots , x_O \in \mathbb{Z}$, then there are only finitely many desired primes. Indeed, as the product of integers having odd order modulo a prime has odd order, all of $\ord_p(o_i)$ being odd would imply $\ord_p(-1)$ being odd, which is not possible for $p > 2$. Thus, we may without loss of generality assume that no product of $o_i$ equals $-1$.

To describe the obstructions of the general case, we need the following lemma \cite[Lemma 5.1]{jarviniemi}.

\begin{lemma}
\label{lemma:oddBasis}
Let $o_1, \ldots , o_O$ be non-zero integers, no product of which equals $-1$. There exists a subset $S$ of $\{o_1, \ldots , o_O\}$ with the following properties.
\begin{itemize}
\item The elements of $S$ are multiplicatively independent (i.e. there is no product of elements of $S$ equal to $1$ except for the empty product).
\item For any $1 \le i \le O$ there exists an odd integer $x$ and a function $f : S \to \mathbb{Z}$ with
$$o_i^x = \prod_{s \in S} s^{f(s)}.$$
\end{itemize}
\end{lemma}

This allows us to reduce to the case when $o_1, \ldots ,o_O$ are multiplicatively independent:

\begin{lemma}
\label{lemma:multInd}
Let $o_1, \ldots , o_O$ be non-zero integers, no product of which equals $-1$. Let $S$ be a set as in Lemma \ref{lemma:oddBasis} and let $p$ be a prime (not dividing any of $o_1, \ldots , o_O$). The orders $\ord_p(o_i)$ are all odd if and only if $\ord_p(s)$ is odd for all $s \in S$.
\end{lemma}

\begin{proof}
``Only if'' is clear. ``If'': Let $1 \le i \le O$ and write
$$o_i^x = \prod_{s \in S} s^{f(s)}$$
with $x$ odd. The order of the right hand side modulo $p$ is odd, so $\ord_p(o_i^x)$ is odd and thus $\ord_p(o_i)$ is odd.
\end{proof}

From now on we assume that the numbers $o_i$ are multiplicatively independent.

Let $A$ be the subset of elements $a \in \{e_1, \ldots , e_E\}$ such that there exist a vector $f(a) = (f(a)_1, f(a)_2, \ldots , f(a)_O)$ of rational numbers satisfying
\begin{align}
\label{eq:eq1}
a = \prod_{i = 1}^O o_i^{f(a)_i}.
\end{align}
By multiplicative independence, this vector is unique. Let $B$ be the rest of $\{e_1, \ldots , e_E\}$.

Obviously if $f(a)$ consists of integers for some $a \in A$, there are only finitely many desired primes, but there are some other obstructions as well.

\section{Results}

The main result gives a complete characterization for the simultaneous insolvability of exponential congruences.

\begin{theorem}
\label{thm:main}
Let $a_i^x \equiv b_i \pmod{p}, 1 \le i \le n$ be a set of exponential congruences. Assume no pair $(a_i, b_i)$ is trivial. As in Section \ref{sec:elementary}, let $o_1, \ldots , o_O$ be the integers whose orders are required to be odd, which may be taken to be multiplicatively independent, and let $e_1, \ldots , e_E$ be the integers whose orders are required to be even. Partition the set $\{e_1, \ldots , e_E\}$ into $A$ and $B$ as in Section \ref{sec:elementary}, and let $f(a)$ be defined as in \eqref{eq:eq1}. Let $M \in \mathbb{Z}_+$ be such that the denominator of $2^Mf(a)_i$ is odd (when written in its lowest terms) for all $a \in A, 1 \le i \le O$.  

There are infinitely many primes $p$ such that none of $a_i^x \equiv b_i \pmod{p}$ are solvable if and only if the system 
$$\sum_{i = 1}^O 2^Mf(a)_ix_i \not\equiv 0 \pmod{2^M}, a \in A$$
of incongruences has an integer solution $(x_1, \ldots , x_O)$.

Furthermore, if there are infinitely many such primes, then their lower density is positive.
\end{theorem}
We therefore see that the simultaneous insolvability of exponential congruences translates to the solvability of simultaneous linear incongruences.

Note that irrational and divisible pairs do not affect the infinitude of the primes at hand.

\begin{corollary}
\label{cor:noEffect}
Let $n$ be a positive integer, and let $a_1, \ldots , a_n, b_1, \ldots , b_n$ be integers. Assume that there are infinitely many primes $p$ such that none of $a_i^x \equiv b_i \pmod{p}, 1 \le i \le n$ are solvable. Let $(a_{n+1}, b_{n+1})$ be a pair of integers which is either irrational or divisible. Then there are infinitely many primes $p$ such that none of $a_i^x \equiv b_i \pmod{p}, 1 \le i \le n+1$ are solvable. Furthermore, the lower density of such primes exists and is positive.
\end{corollary}

Another corollary is that for positive $a_i, b_i$, the only cases when there are only finitely many desired primes are the trivial ones.

\begin{corollary}
\label{cor:positive}
Let $a_1, \ldots , a_n, b_1, \ldots , b_n$ be integers greater than $1$. Assume that $b_i$ is not a power of $a_i$ for any $i$. There are infinitely many primes $p$ such that none of the equations $a_i^{x_i} \equiv b_i \pmod{p}$ are solvable. Furthermore, the lower density of such primes exists and is positive.
\end{corollary}

\begin{proof}
Note that none of the pairs $(a_i, b_i)$ are odd, so the conditions of Theorem \ref{thm:main} are trivially satisfied as long as no pair $(a_i, b_i)$ is trivial.
\end{proof}

We then make a couple of comments on the system in Theorem \ref{thm:main}. The numbers $f(a)_i2^M$ are rational numbers whose denominators are odd, so they may be viewed as integers modulo $2^M$. However, they are not necessarily zero modulo $2^M$ due to cancellations, e.g. if $f(a)_i = 1/2^M$.

As an example, if $\ord_p(4)$ and $\ord_p(9)$ are required to be odd, and $\ord_p(2), \ord_p(3)$ and $\ord_p(6)$ are required to be even, we obtain the system
$$x_1 \not\equiv 0, x_2 \not\equiv 0, x_1 + x_2 \not\equiv 0 \pmod{2}.$$
This system has no solutions, and hence there are only finitely many desired primes. The idea is that if $p$ is a prime with $v_2(p-1) = k$, then $\ord_p(4)$ and $\ord_p(9)$ being odd implies that $2$ and $3$ are perfect $2^{k-1}$th powers modulo $p$. One sees (cf. multiplicativity of Legendre's symbol) that this implies that at least one of $2, 3$ and $6$ is a perfect $2^k$th power modulo $p$, implying that at least one of $\ord_p(2), \ord_p(3), \ord_p(6)$ is odd.

If one can choose $M = 1$, one obtains a system of linear equations over the finite field of two elements, namely that some linear combinations of $x_i$ should equal $1$. The solvability of such a system is well understood: there exists a solution if and only if there is no odd number of the linear combinations of $x_i$ whose sum vanishes.

For higher values of $M$ it seems that there is no as simple of a criterion for the existence of a solution. Already the solvability of a system of linear equations is more difficult over rings than over fields. The solvability of the system of incongruences in Theorem \ref{thm:main} corresponds to the solvability of at least one of $(2^M - 1)^{|A|}$ systems of linear congruences.

We then prove the following result on divisibility conditions imposed on orders.

\begin{theorem}
\label{thm:gcd}
Let $a_i, g_i, m_i, 1 \le i \le n$ be non-zero integers with $1 \le g_i \mid m_i$ for all $i$. The following are equivalent.
\begin{itemize}
\item[(i)] There are infinitely many primes $p$ such that $\gcd(\ord_p(a_i), m_i) = g_i$ for all $i$.
\item[(ii)] For any prime $q$ the following holds: There are infinitely many primes $p$ such that $\gcd(\ord_p(a_i), q^{v_q(m_i)}) = q^{v_q(g_i)}$ for all $i$.
\end{itemize}
Furthermore, if there are infinitely many such primes, their density exists and is positive.
\end{theorem}
Note that the condition in (ii) is interesting for only finitely many $q$. One can give a characterization for the satisfiability of these conditions similarly as in Theorem \ref{thm:main} in terms of systems of linear incongruences, but we do not state this result explicitly here.

As seen from Theorem \ref{thm:main}, the necessary and sufficient condition for the satisfiability of (in)divisibility conditions is not a very simple one. However, we do have the following result.

\begin{theorem}
\label{thm:div}
Let $a_i, m_i, 1 \le i \le n$ be non-zero integers with $|a_i| > 1$ for all $i$. There are infinitely many primes $p$ such that $m_i \mid \ord_p(a_i)$ for all $i$. Furthermore, the density of such primes exists and is positive.
\end{theorem}

See the comment below for the existence of the density. Here we prove only that the lower density is positive.

\begin{proof}
Note that $\ord_p(a)$ and $\ord_p(-a)$ differ by a factor of $1/2, 1$ or $2$. Thus, it suffices to guarantee $2m_i \mid \ord_p(|a_i|)$, i.e. we may assume $a_i > 1$ for all $i$.

Let $q$ be a prime and $k \in \mathbb{Z}_+$. The equation
$$(a_i^{q^k})^x \equiv a_i^{q^{k-1}} \pmod{p}$$
is insolvable if and only if $q^k \mid \ord_p(a_i)$. By Corollary \ref{cor:positive} one may simultaneously satisfy any number of such conditions as long as $a_i > 1$ for a set of primes with positive lower density.
\end{proof}

We also mention the following result on the indivisibilities of orders. Note that by Theorem \ref{thm:gcd} we lose no generality by considering indivisibility by a single prime at a time.

\begin{theorem}
\label{thm:inDiv}
Let $a_1, \ldots , a_n$ be non-zero integers and let $q$ be a prime.
\begin{itemize}
\item[(i)] If $q$ is odd, then there are infinitely many primes $p$ such that $q \nmid \ord_p(a_i)$ for all $i$.
\item[(ii)] If $q = 2$, then there are infinitely many primes $p$ such that $q \nmid \ord_p(a_i)$ for all $i$ if and only if there does not exist integers $e_1, \ldots , e_n$ such that
$$\prod_{i = 1}^O a_i^{e_i} = -1.$$
\end{itemize}
Furthermore, if there are infinitely many such primes, then their density exists and is positive.
\end{theorem}
(The case $q = 2$ is implicitly given by Theorem \ref{thm:main} and the discussion in Section \ref{sec:elementary}.)

In Theorems \ref{thm:gcd}, \ref{thm:div} and \ref{thm:inDiv} one can prove the existence of the density. The idea is that one can express the relevant set of primes as a countable disjoint union of sets of primes with suitable Artin symbols (by controlling divisors of $p-1$ and how perfect powers $a_i$ are modulo $p$). The density of such sets tends to zero, so one can apply the Chebotarev density theorem to a finite number of them to obtain arbitrarily good approximations for the density. We omit a detailed proof, but the reader may find an execution of this idea in a slightly easier case in \cite{lagarias}. See also \cite[Section 8.2]{moree} for more references on divisibilities of orders.

\section{Notation and conventions} \label{sec:notations}

The letter $p$ denotes a (rational) prime.  For an integer $x$ not divisible by $p$ the order $\ord_p(x)$ of $x$ modulo $p$ is the smallest positive integer $e$ such that $x^e \equiv 1 \pmod{p}$. For $x \neq 0$ we denote by $v_p(x)$ the largest $e$ such that $p^e | x$.

By $\zeta_k$ we denote a primitive $k$th root of unity. 

For a Galois extension $K$ of $\mathbb{Q}$ we denote by $\Gal(K/\mathbb{Q})$ its Galois group. For an unramified prime $p$ the Artin symbol of $p$ with respect to $K$ is denoted by
$$\left(\frac{K/\mathbb{Q}}{p}\right).$$

We use the fact that an unramified $p$ splits completely in $K$ if and only if $\left(\frac{K/\mathbb{Q}}{p}\right)$ is the identity element of $\Gal(K/\mathbb{Q})$. If $K \subset L$, then the restriction of $\left(\frac{L/\mathbb{Q}}{p}\right)$ to $\Gal(K/\mathbb{Q})$ is $\left(\frac{K/\mathbb{Q}}{p}\right)$.

In particular, the Artin symbol of a prime $p$ with respect to a extension such as $\mathbb{Q}(\zeta_n, a^{1/n})$ controls the remainder of $p$ modulo $n$ (via the image of the root of unity $\zeta_n$) and whether $a$ is a $d$th power modulo $p$ or not for all $d \mid n$ (via the image of the element $a^{1/n}$).

We use the following version of the Chebotarev density theorem.

\begin{theorem}[Chebotarev density theorem]
Let $K/\mathbb{Q}$ be a finite Galois extension with Galois group $G$, and let $C$ be a conjugacy class of $G$. Then, the set 
$$S = \{p | p \text{ is unramified in } K \text{ and } \left(\frac{K/\mathbb{Q}}{p}\right) = C\}$$
has natural density $\frac{|C|}{|G|}.$
\end{theorem}

As our proofs of infinitude of sets of primes is based on the Chebotarev density theorem, this will automatically lead to a positive lower density for the set at hand.

\section{Background on Kummer extensions}\label{sec:kummer}

We state the following results from \cite[Section 3]{jarviniemi}. More general results may be found in \cite{perucca-sgobba}.

\begin{proposition}
\label{prop:kummerGalois}
Let $a_1, \ldots , a_k$ be multiplicatively independent rationals, and let $K$ be a finite Galois extension of $\mathbb{Q}$. There exists a positive integer $N$ with the follwing property:

For any integers $n, m_1, \ldots , m_k$, where $m_i \mid n$ for all $i$, and $x, x_1, \ldots , x_k$ with $(x, Nn) = 1, N \mid x-1, x_1, \ldots , x_k$ there exists an element of the Galois group of
$$K(\zeta_{Nn}, a_1^{1/Nm_1}, \ldots , a_k^{1/Nm_k})/K$$
sending
$$\zeta_{Nn} \to \zeta_{Nn}^{x}, a_i^{1/Nm_i} \to \zeta_{Nm_i}^{x_i}a_i^{1/Nm_i}.$$
\end{proposition}

\begin{proposition}
\label{prop:kummerDisjoint}
Let $a_1, \ldots , a_k$ be multiplicatively independent rationals, and let $K$ be a finite Galois extension of $\mathbb{Q}$. There exists an integer $N$ such that for any $n, n', m_1, \ldots m_k$, where $(n, Nn') = 1$ and $m_i \mid n$ for all $i$, the fields
$$\mathbb{Q}(\zeta_n, a_1^{1/m_1}, \ldots , a_k^{1/m_k})$$
and
$$K(\zeta_{Nn'}, a_1^{1/Nn'}, \ldots , a_k^{1/Nn'}).$$
are linearly disjoint and the former extension has degree $\phi(n)m_1 \cdots m_k$.
\end{proposition}
Recall that finite Galois extensions $K_1$ and $K_2$ are linearly disjoint (over $\mathbb{Q}$) iff one has the isomorphism $\Gal(K_1K_2/\mathbb{Q}) \cong \Gal(K_1/\mathbb{Q}) \times \Gal(K_2/\mathbb{Q})$.

\section{Proof of Theorem \ref{thm:inDiv}}
\label{sec:inDiv}

We first prove Theorem \ref{thm:inDiv} since it is the easiest one. Assume first that $q \neq 2$.

Let $\ell_1, \ldots , \ell_m$ be the primes which divide some of $a_i$. Let $N$ be as in Proposition \ref{prop:kummerGalois} when applied to $\ell_1, \ldots , \ell_m$ (with the base field $K = \mathbb{Q}$). One sees that there exists an integer $x$ such that $Nq \mid x-1$ and $Nq^2 \nmid x-1$. Clearly now $\gcd(Nq^2, x) = 1$.

By the choice of $N$, there exists an automorphism $\sigma$ of
$$K = \mathbb{Q}(\zeta_{Nq^2}, \ell_1^{1/Nq}, \ldots , \ell_m^{1/Nq})$$
such that
$$\sigma(\zeta_{Nq^2}) = \zeta_{Nq^2}^x, \sigma(\ell_i^{1/Nq}) = \ell_i^{1/Nq}.$$
By the Chebotarev density theorem, there exists infinitely many primes $p$ such that $\sigma \in \left(\frac{K/\mathbb{Q}}{p}\right)$. 

For these primes $p$ one has $p \equiv x \pmod{Nq^2}$, so $p \equiv 1 \pmod{Nq}$ but $p \not\equiv 1 \pmod{Nq^2}$, and $\ell_i$ is a perfect $Nq$th power modulo $p$. Thus, each of $|a_i|$ is a perfect $Nq$th power modulo $p$, and so one has $q \nmid \ord_p(|a_i|)$. Since $q$ is odd and $\ord_p(a)$ and $\ord_p(-a)$ differ by only a power of two, one has $q \nmid \ord_p(a_i)$.

If $q = 2$, assume that no product of $a_i$ equals $-1$. One may apply Lemma \ref{lemma:multInd} to assume that the numbers $a_i$ are multiplicatively independent. One then repeats the above argument with the numbers $\ell_1, \ldots , \ell_m$ replaced with the multiplicatively independent numbers $a_1, \ldots , a_n$.

\section{Necessity of the conditions of Theorem \ref{thm:main}} \label{sec:necessity}

Assume that there are infinitely many primes $p$ satisfying the conditions. Let $p$ be such a prime which is larger than all of $|o_i|$ and $|e_i|$, and let $k = v_2(p-1)$. Now $o_i$ is a perfect $2^k$th power modulo $p$ for all $i$, while the numbers $e_i$ are not.

Let $g$ be a primitive root modulo $p$. Let $q_1, \ldots , q_m$ be the primes which divide at least one of $o_1, \ldots , o_O$. Let $\ell(q_i)$ be an integer such that $g^{\ell(q_i)} \equiv q_i \pmod{p}$. Define
$$\ell(q_1^{e_1} \cdots q_m^{e_m}) = e_1\ell(q_1) + \ldots + e_m\ell(q_m),$$
where $e_i$ are any (not necessarily positive) integers.

Now $\ell(o_i), 1 \le i \le O$ and $\ell(a), a \in A$ are defined. Furthermore $g^{\ell(x)} \equiv x \pmod{p}$ for all $x \in \mathbb{Z}$ such that $\ell(x)$ is defined, and therefore $2^k \mid \ell(o_i)$ for all $i$ and $2^k \nmid \ell(a)$ for all $a \in A$.

Let $a \in A$ be arbitrary. Let $M'$ be an odd integer such that $2^MM'f(a)_i \in \mathbb{Z}$ for all $i$. Now
$$2^MM'\ell(a) = \ell(a^{2^MM'}) = \ell\left(\prod_{i = 1}^O o_i^{2^MM'f(a)_i}\right) = \sum_{i = 1}^O 2^MM'f(a)_i \ell(o_i).$$
The left hand side is not divisible by $2^{M+k}$, implying that
$$\sum_{i = 1}^O 2^MM'f(a)_i\ell(o_i) \not\equiv 0 \pmod{2^{M+k}}.$$
Hence
$$\sum_{i = 1}^O \frac{\ell(o_i)}{2^k} (f(a)_i2^M) \not\equiv 0 \pmod{2^M},$$
so the numbers $\ell(o_i)/2^k$ give a solution to the system of incongruences.

\section{Sufficiency of the conditions of Theorem \ref{thm:main}}
\label{sec:sufficiency}

Here is the idea of the proof. The parity conditions are handled by using a solution to the system in Theorem \ref{thm:main} to decide on power residue conditions (cf. Section \ref{sec:necessity}, where the power residue conditions were used to obtain a solution to the system). The divisibility of orders is based on requiring $p \equiv 1 \pmod{q^k}$ for various primes $q$, where $k$ is large. This leads to $q \mid \ord_p(c)$ with ``high probability''. Similarly, irrational pairs $(a, b)$ are controlled by taking a large prime $q$ and requiring $a$ to be a perfect $q$th power modulo $p \equiv 1 \pmod{q}$, and one has $b$ not a perfect $q$th power with high probability. The imposed conditions can all be satisfied simultaneously by the tools in Section \ref{sec:kummer} and the Chebotarev density theorem.

Expand the basis $\{o_1, \ldots , o_O\}$ of $A$ into a $\mathbb{Q}$-basis for $A \cup B$ by a subset of the elements of $B$. Let $S_1$ be the constructed set. Thus, each element of $E$ may be expressed uniquely in the form
$$\prod_{s \in S_1} s^{f(s)},$$
where $f : S_1 \to \mathbb{Q}$.

Let $(c_i, q_i), 1 \le i \le n_1$ denote the pairs of integers corresponding to divisible pair for which we require $q_i \mid \ord_p(c_i)$. Expand the basis $S_1$ of $E$ to a basis $S_2$ of $E \cup \{c_1, \ldots , c_{n_1}\}$.

Let $N$ be as in Proposition \ref{prop:kummerGalois} when applied to the elements of $S_2$ (with the field $K = \mathbb{Q}$). Let $T$ denote two times the product of the distinct primes in $\{q_1, \ldots , q_{n_1}\}$.

We now prove the existence of infinitely many primes $p$ with $2 \nmid \ord_p(o_i), 2 \mid \ord_p(e_i), q_i \mid \ord_p(c_i)$. 

Let  $k$ be an arbitrarily large positive integer and let $x_1, \ldots , x_O$ be a solution to the system of incongruences modulo $2^M$. Construct the function $x : S_2 \to \mathbb{Z}$ as follows.
\begin{itemize}
\item For $o_i \in S_2$, choose $x(o_i)$ to be a uniformly random integer from $[1, 2^MT^k]$ satisfying $x(o_i) \equiv 2^kx_i \pmod{2^{M+k}}$.
\item For the elements $u \in S_2 \setminus \{o_1, \ldots , o_O\}$, choose $x(u)$ uniformly at random from $[1, 2^MT^k]$.
\end{itemize}
Also, let $X = NT^k + 1$.

Consider the automorphism $\sigma_x$ of
$$K_k = \mathbb{Q}(\zeta_{2^{M+1}NT^k}, S_2^{1/2^MNT^k}),$$
where $S^{1/n} = \{s^{1/n}, s \in S\}$, sending
$$\zeta_{2^{M+1}NT^k} \to \zeta_{2^{M+1}NT^k}^X, s_2^{1/2^MNT^k} \to \zeta_{2^MNT^k}^{Nx(s_2)}s_2^{1/2^MNT^k}$$
for all $s_2 \in S_2$. This $\sigma_x$ is well defined by the choice of $N$.

By the Chebotarev density theorem, there are infinitely many primes $p$ whose Artin symbol with respect to $K_k$ contains $\sigma_x$. Consider these primes $p$. We make the following five observations.

\begin{itemize}
\item[(i)] $p \equiv X \pmod{2^{M+1}NT^k}$, so $p \equiv 1 \pmod{NT^k}$ and $v_2(p-1) = v_2(N) + k$.
\item[(ii)] By the choice of $x(o_i)$, the element $o_i^{1/2^{v_2(N) + k}}$ is fixed. Thus, $o_i$ is a perfect $2^{v_2(N) + k}$th power modulo $p$, and hence $\ord_p(o_i)$ is odd.
\item[(iii)] For any $a \in A$, let $M' = M'(a)$ be an odd integer such that $2^MM'f(a)_i$ is an integer. The number 
$$a^{2^MM'/2^{v_2(N) + k+M}} =  \prod_{i = 1}^O o_i^{2^MM'f(a)_i/2^{v_2(N) + k+M}}$$
is an element of $K_k$ which is mapped to
$$a^{2^MM'/2^{v_2(N) + k + M}} \prod_{i = 1}^O \zeta_{2^{v_2(N) + k + M}}^{2^MM'f(a)_iNx(o_i)}.$$
By the choice of $x(o_i)$, the sum
$$\sum_{i = 1}^O 2^MM'f(a)_iNx(o_i)$$
is not zero modulo $2^{M+v_2(N)+k}$. Therefore the element
$$a^{2^MM'/2^{v_2(N) + k + M}} = a^{M'/2^{v_2(N) + k}}$$
is not fixed, and hence $a^{M'}$ is not a perfect $2^{v_2(N) + k}$th power modulo $p$. This implies $2 \mid \ord_p(a)$.

\item[(iv)] For any $b \in B$, write
$$b = \prod_{s \in S_1} s^{f(s)},$$
where $f : S_1 \to \mathbb{Q}$. Here $f(s) \neq 0$ for at least one $s \in S_1 \setminus \{o_1, \ldots , o_O\}$. We claim that $2 \mid \ord_p(b)$ with probability approaching $1$ as $k \to \infty$ (with respect to the random choice of $\sigma_x$).

Let $M' = M'(b) \in \mathbb{Z}_+$ be such that all of $M'f(s)$ are integers. Consider the element
$$b^{M'/2^{v_2(N) + k}} = \prod_{s \in S_1} s^{M'f(s)/2^{v_2(N) + k}}$$
of $K_k$ and its image
$$b^{M'/2^{v_2(N) + k}} \prod_{s \in S_1} \zeta_{2^{k+v_2(N)}}^{M'f(s)Nx(s)}$$
under $\sigma_x$. Since $f(s) \neq 0$ for at least one $s \in S_1 \setminus \{o_1, \ldots , o_O\}$ and this $f(s)$ is random modulo $2^k$, as $k \to \infty$ the probability that
$$\sum_{s \in S_1} M'f(s)Nx(s) \equiv 0 \pmod{2^{k+v_2(N)}}$$
approaches $0$. Thus, the probability of choosing $\sigma_x$ such that $b^{M'}$ is a perfect $2^{k + v_2(N)}$th power modulo $p$ approaches zero. This means that there are ``many'' choices of $\sigma_x$ such that $\ord_p(b)$ is even for the primes $p$ with $\sigma_x \in \left(\frac{K_k/\mathbb{Q}}{p}\right)$.

\item[(v)] Let $(c_i, q_i), 1 \le i \le n_1$ be some pair corresponding to a divisible pair, so we want $q_i \mid \ord_p(c_i)$. We claim that this happens with probability approaching $1$ as $k \to \infty$. The proof is similar to the one in (iv): Write
$$c_i = \prod_{s \in S_2} s^{f(s)},$$
where $f : S_2 \to \mathbb{Q}$. Let $M' = M'(c_i)$ be such that $M'f(s)$ is an integer. The element $c_i^{M'/q_i^{v_q(N) + k}}$ is fixed under $\sigma_x$ if and only if
$$\sum_{s \in S_2} M'f(s)Nx(s) \equiv 0 \pmod{q_i^{v_q(N) + k}}.$$
As the numbers $x(s)$ are random modulo $q_i^{k}$, this happens with probability approaching $0$, so $\ord_p(c_i^{M'})$ and hence $\ord_p(c_i)$ are divisible by $q_i$ with high probability.
\end{itemize}

Hence there exists some $k$ and $\sigma_x$ such that the primes $p$ with $\sigma_x \in \left(\frac{K_k/\mathbb{Q}}{p}\right)$ satisfy the conditions $2 \nmid \ord_p(o_i), 2 \mid \ord_p(e_i), q_i \mid \ord_p(c_i)$. We are left with handling the irrational pairs.

Let $(a_1, b_1), \ldots , (a_{n_2}, b_{n_2})$ be the irrational pairs. For each $1 \le i \le n_2$, pick a prime $q_i$ such that the fields
$$L_i = \mathbb{Q}(\zeta_{q_i}, a_i^{1/q_i}, b_i^{1/q_i})$$
and the field $K_k$ constructed above are linearly disjoint and such that the degree of $L_i$ is the maximum possible $q_i^2(q_i - 1)$ for all $i$. The existence of such primes $q_i$ is guaranteed by Proposition \ref{prop:kummerDisjoint}. (In fact, any choice of large enough distinct primes works.)

For each $L_i$ there exists an element of $\Gal(L_i/\mathbb{Q})$ fixing $\zeta_{q_i}$ and $a_i^{1/q_i}$ but which does not fix $b_i^{1/q_i}$. The primes $p$ with the corresponding Artin symbol are such that $p \equiv 1 \pmod{q_i}$, $a_i$ is a $q_i$th power modulo $p$ and $b_i$ is not. This leads to the insolvability of $a_i^x \equiv b_i \pmod{p}$.

We have already proved the existence of an element of $\Gal(K_k/\mathbb{Q})$ taking care of parities and divisibilities of orders. By linear disjointness,
$$\Gal(K_kL_1 \cdots L_{n_2}/\mathbb{Q}) \cong \Gal(K_k/\mathbb{Q}) \times \Gal(L_1/\mathbb{Q}) \times \cdots \times \Gal(L_{n_2}/\mathbb{Q}).$$
We may therefore merge the constructed maps on $K_k, L_1, \ldots , L_{n_2}$ to an automorphism of the compositum $K_kL_1 \cdots L_{n_2}$. The infinitely many primes $p$ with the corresponding Artin symbol satisfy the conditions of Theorem \ref{thm:main}.

\section{Proof of Theorem \ref{thm:gcd}}

Clearly (i) implies (ii), so we focus on the other direction.

The idea is roughly as follows. Let $\ell_1, \ldots, \ell_m$ be the primes which divide at least one of $a_1, \ldots , a_n$. For each prime $q \mid m_1 \cdots m_n$ take a large prime $p$, and consider which of the numbers of the form $\ell_1^{e_1} \cdots \ell_m^{e_m}, e_i \in \mathbb{Z}$ are perfect $q^k$th powers modulo $p$ for $k = 1, 2, \ldots$ This tells us how we should choose $q^k$th power residues modulo $p$ in order to guranatee $\gcd(\ord_p(a_i), q^{v_q(m_i)}) = q^{v_q(g_i)}$. We then prove that one may combine these conditions.

Fix some prime $q \mid m_1 \cdots m_n$, and let $p \nmid 2qa_1 \cdots a_n$ be a prime such that $q^{v_q(g_i)} \mid \ord_p(a_i)$ and $q^{v_q(g_i) + 1} \nmid \ord_p(a_i)$ when $q \mid m_i/g_i$. Denote $k = v_q(p-1)$.  Now $p$ is unramified in
$$K_{q, k} = \mathbb{Q}(\zeta_{2q^{k+1}}, \ell_1^{1/q^k}, \ldots , \ell_m^{1/q^k}).$$
Let $C$ denote the Artin symbol of $p$ with respect to $K_{q, k}$, and let $\sigma_{q, k}$ be any of its elements. Let $\sigma_{q, k}$ map
$$\zeta_{2q^{k+1}} \to \zeta_{2q^{k+1}}^x, \ell_i^{1/q^k} \to  \zeta_{q^k}^{x_i}\ell_i^{1/q^k}.$$
By the choice of $p$ we have 
\begin{align}
\label{eq:1}
v_q(x-1) = k,
\end{align}
and by the divisibility conditions on orders
\begin{align}
\label{eq:2}
v_q\left((x-1)\epsilon_j + \sum_{i = 1}^m v_{\ell_i}(a_j)x_i\right) \le k - v_q(g_j)
\end{align}
for all $1 \le j \le n$, equality occuring at least when $q \mid m_j/g_j$. Here $\epsilon_j = 0$ if $a_j > 0$ and $\epsilon_j = 1/2$ if $a_j < 0$. This term is present since for $a_j < 0$ we have
$$\sigma_{q, k}(a_j^{1/q^k}) = \zeta_{2q^k}^{x} \sigma_{q, k}(|a_j|^{1/q^k}) = \zeta_{2q^k}^{x} \zeta_{q^k}^{\sum_{i = 1}^m v_{\ell_i}(a_j)x_i} |a_j|^{1/q^k} = \zeta_{q^k}^{(x-1)/2 + \sum_{i = 1}^m v_{\ell_i}(a_j)x_i}a_j^{1/q^k}.$$

The transformation $x-1 \to q(x-1)$, $x_i \to qx_i, k \to k+1$ does not affect the truth of \eqref{eq:1} and \eqref{eq:2}. We deduce that for any prime $q$ and positive integer $k$, there exists integers $x, x_1, \ldots , x_m$ satisfying $x > 1, q^k \mid x-1, x_1, \ldots , x_m$ and
\begin{align}
\label{eq:3}
v_q\left((x-1)\epsilon_j + \sum_{i = 1}^m v_{\ell_i}(a_j)x_i\right) \le v_q(x-1) - v_q(g_j),
\end{align}
again with equality when $q \mid m_j/g_j$.

Let $N$ be as in Proposition \ref{prop:kummerGalois} when applied to the numbers $\ell_1, \ldots , \ell_m$. Let $T$ be the product of all primes dividing at least one of $m_1, \ldots , m_n$. By the Chinese remainder theorem there exist integers $x, x_1, \ldots , x_m$ satisfying $x > 1,$ $N \mid x-1, x_1, \ldots , x_m$, and \eqref{eq:3} for all $q \mid T, 1 \le j \le n$ (again with correct equality cases). We may additionally assume $\gcd(x, 2NT) = 1$. Let
$$P = \prod_{q \mid T} q^{v_q(x-1)+1}.$$

By the choice of $N$, there exists an automorphism $\sigma$ of
$$K = \mathbb{Q}(\zeta_{2NP}, \ell_1^{1/NP}, \ldots , \ell_m^{1/NP})$$
mapping
$$\zeta_{2NP} \to \zeta_{2NP}^x, \ell_i^{1/NP} \to \zeta_{NP}^{x_i}\ell_i^{1/NP}.$$

Apply the Chebotarev density theorem. Let $p$ be a prime whose Artin symbol with respect to $K$ contains $\sigma$. From \eqref{eq:3} one now sees that $\ord_p(a_j)$ is divisible by $q^{v_q(g_j)}$ for all $q \mid T, 1 \le j \le n$, and not by $q^{v_q(g_j)+1}$ for $q \mid m_j/g_j$.

\begin{remark}
The degree of
$$\mathbb{Q}(\zeta_t, \ell_1^{1/t}, \ldots , \ell_m^{1/t})$$
is not in general $\phi(t)t^m$ due to square roots of integers lying in cyclotomic fields. (In fact, this is the only reason for non-maximality, and at least for $t$ odd the degree is $\phi(t)t^m$.) Since the degree is not maximal, there are some restrictions on the images of $\zeta_t$ and $\ell_i^{1/t}$ under the elements of the Galois group. However, the degree is almost maximal by Proposition \ref{prop:kummerGalois} (in this case the degree is at least $\phi(t)t^m/2^m$), so by repeatedly performing the transformation $x-1 \to q(x-1), x_i \to qx_i, k \to k+1$ in the proof we get away from these ``low-level'' restrictions on the elements of the Galois group.
\end{remark}

\bibliography{insolvabilityExponential}
\bibliographystyle{plain}

\end{document}